\documentclass[12pt, reqno, twoside, letterpaper]{amsart}

\usepackage{paperstyle}
\usepackage{graphicx}

\newcommand{\be}{\begin{equation}}
\newcommand{\ee}{\end{equation}}
\newcommand{\dalign}[1]{\[\begin{aligned} #1 \end{aligned}\]}

\usepackage{mathtools}
\usepackage{todonotes}
\usepackage[norefs,nocites]{refcheck}

\title[Arithmetic progressions of Carmichael numbers]
{Arithmetic progressions of Carmichael numbers
in a reduced residue class}
      
\author[W.~D.~Banks]{William Banks}

\address{Department of Mathematics, 
         University of Missouri, 
         Columbia MO 65211 USA.}

\email{bankswd@missouri.edu}

\date{\today}

\begin{document}

\begin{abstract}
Fix coprime natural numbers $a,q$.
Assuming the Prime $k$-tuple Conjecture, we show that
there exist arbitrarily long arithmetic progressions of
Carmichael numbers, each of which lies in the reduced
residue class $a$~mod~$q$ and is a product
of three distinct prime numbers.
\end{abstract}


\maketitle



{\Large\section{Introduction}}

For any prime number $n$, Fermat's little theorem asserts that
\be\label{eq:FLT}
x^n\equiv x\bmod n\qquad (x\in\Z).
\ee
Around 1910, Carmichael initiated the study of
\emph{composite} numbers~$n$ with the same property; these
integers are now called \emph{Carmichael numbers}.  In 1994
the existence of infinitely many Carmichael numbers was 
established by Alford, Granville and Pomerance~\cite{AGP};
see also~\cite{GranPom}.

Since both primes and Carmichael numbers share the property
\eqref{eq:FLT},  it seems natural to ask whether certain
known results about primes can also be proved for Carmichael numbers,
and indeed this theme has been explored by several authors.
For example, the analogue of \emph{Dirichlet's theorem} about the
infinitude of primes in a reduced residue class has been
established for  Carmichael numbers by Wright~\cite{Wright},
building on ideas of Banks and Pomerance~\cite{BanksPomer}
and of Matom\"aki~\cite{Mato}.

In 2008 a stunning and celebrated work of Green and Tao~\cite{GreenTao}
established the existence of arbitrarily long arithmetic progressions
in the primes; their landmark paper in additive number theory at once
resolved the longstanding open problem about prime numbers and also
an important case of a famous conjecture of Erd\H os on arithmetic
progressions. In the present note, we give a conditional proof of a
similar result for Carmichael numbers.

\bigskip

\begin{theorem}\label{thm:main}
Let $a,q$ be fixed coprime natural numbers.
Under the Prime $k$-tuple Conjecture, there exist arbitrarily long
arithmetic progressions of Carmichael numbers, each of which
lies in the reduced residue class $a$~mod~$q$ and is a product
of three distinct prime numbers.
\end{theorem}

In light of this result, we conjecture that every reduced residue
class $a$~mod~$q$ contains arbitrarily long arithmetic
progressions of Carmichael numbers. 

\bigskip

\subsection{Acknowledgements.}
The original draft of this manuscript established (under the Prime
$k$-tuple Conjecture) the existence of arbitrarily long
arithmetic progressions of Carmichael numbers. The author thanks
Andrew Granville for sharing a simpler proof (elements of which
are incorporated here) and for posing the question as to whether
the same result holds true in an arbitrary reduced residue
class $a$~mod~$q$.

{\Large\section{The Prime $k$-tuple Conjecture}}

A $k$-tuple of linear forms in $\ZZ[X]$, denoted by
$$
\cH(X) = \{g_jX + h_j\}_{j=1}^k, 
$$
is said to be \emph{admissible} if the associated polynomial
$f_\cH(X)\defeq\prod_{j=1}^k(g_jX + h_j)$ 
has no fixed prime divisor, that is, if 
$$
\big|\{n \bmod p : f_\cH(n) \equiv 0 \bmod p\}\big|\ne p
$$
for every prime  $p$.
In this note we consider only $k$-tuples for which
\be\label{eq:vandermonde}
g_1,\ldots,g_k > 0 
\mand
\prod_{1 \le i < j\le k}(g_ih_j - g_jh_i) \ne 0.
\ee
The \emph{Prime $k$-tuple Conjecture} asserts that if 
$\cH(X)$ is admissible and satisfies~\eqref{eq:vandermonde}, then 
$
\cH(n) = \{g_jn + h_j\}_{j=1}^k
$
is a $k$-tuple of primes for infinitely many $n \in \N$.

\bigskip

{\Large\section{Proof of Theorem~\ref{thm:main}}}

\begin{lemma}\label{lem:Carm-criterion}
Suppose  $b,c,d,e\in\N$ satisfy the conditions
\begin{gather}
\label{eq:crazy}
b,c,d\text{~are coprime in pairs},\\
\label{eq:that}
cde+c+d\equiv 0\bmod b,\\
\label{eq:darn}
bde+b+d\equiv 0\bmod c,\\
\label{eq:cat}
bce+b+c\equiv 0\bmod d.
\end{gather}
If $n\equiv e\bmod bcd$ and the numbers
$$
r\defeq bn+1,\qquad
s\defeq cn+1\mand
t\defeq dn+1
$$
are distinct primes, then $rst$ is a Carmichael number.
\end{lemma}

\begin{proof}
Let $\lambda$ denote the Carmichael function. A composite number $N$
is a Carmichael number if and only if $\lambda(N)\mid N-1$.

Put $N\defeq rst$.
Using \eqref{eq:crazy} we have $\lambda(N)=bcdn$, which divides
$$
N-1=bcdn^3+(bc+bd+cd)n^2+(b+c+d)n
$$
if and only if
$$
bcd\mid(bc+bd+cd)n+b+c+d.
$$
Since $n\equiv e\bmod bcd$, the last condition follows
from \eqref{eq:that}--\eqref{eq:cat}.
\end{proof}

\begin{proof}[Proof of Theorem~\ref{thm:main}]
For any prime $p$, let $v_p$ be the standard $p$-adic valuation.

Let $a,q$ be fixed coprime natural numbers,
and let $\hat a$ be an integer such that
$a\hat a\equiv 1\bmod q$. For any prime $p\mid q$ let
$$
\alpha_p\defeq v_p(q),\qquad
\beta_p\defeq v_p(a-1),\qquad
\hat\beta_p\defeq v_p(\hat a-1),
$$
and put
$$
q_\sharp
\defeq\sprod{p\,\mid\,q\\\beta_p<\alpha_p}p^{\alpha_p},\qquad
q_\flat
\defeq\sprod{p\,\mid\,q\\\beta_p\ge\alpha_p}p^{\alpha_p}.
$$
Clearly,
\be\label{eq:Leo}
q=q_\sharp q_\flat,\qquad
\gcd(q_\sharp,q_\flat)=1,\qquad
a\equiv \hat a\equiv 1\bmod q_\flat.
\ee
From the theory of valuations it is immediate that
$\beta_p=\hat\beta_p$ for any prime
$p\mid q_\sharp$. We define
$$
n_p\defeq\frac{a-1}{p^{\beta_p}}\in\N\mand
\hat n_p\defeq\frac{\hat a-1}{p^{\beta_p}}\in\N
$$
for such primes, and so we have $\gcd(n_p\hat n_p,p)=1$.

Let $e$ be a natural number such that
\be\label{eq:cong1}
e\equiv p^{\beta_p}\bmod p^{\alpha_p}\qquad
(p\mid q_\sharp);
\ee
such integers $e$ exist by the Chinese Remainder Theorem (CRT).
We define
$$
e_p\defeq\frac{e}{p^{\beta_p}}\in\N\qquad(p\mid q_\sharp).
$$
Let $b,c,d$ be large and distinct primes for which the congruences
\dalign{
b\,e_p&\equiv n_p\bmod p^{\alpha_p-\beta_p},\\
c\,e_p&\equiv n_p\bmod p^{\alpha_p-\beta_p},\\
d\,e_p&\equiv \hat n_p\bmod p^{\alpha_p-\beta_p},
}
hold for all primes $p\mid q_\sharp$ (such $b,c,d$
exist because $\gcd(e_pn_p\hat n_p,p)=1$).
After multiplying the above congruences by $p^{\beta_p}$, we get that
\be\label{eq:ball}
\begin{split}
b\,e+1&\equiv a\bmod p^{\alpha_p},\\
c\,e+1&\equiv a\bmod p^{\alpha_p},\\
d\,e+1&\equiv \hat a\bmod p^{\alpha_p},
\end{split}
\ee
for all $p\mid q_\sharp$.
Moreover, taking into account \eqref{eq:Leo}, the same
congruences hold for any prime
$p\mid q_\flat$ provided that
\be\label{eq:cong2}
e\equiv 0\bmod q_\flat.
\ee
Assuming \eqref{eq:cong1} and \eqref{eq:cong2}, by the CRT
and \eqref{eq:ball} it follows that
\be\label{eq:eureka}
(b\,e+1)(c\,e+1)(d\,e+1)\equiv a\bmod q.
\ee

In the above construction, we choose the three primes $b,c,d$ large
enough so that each one exceeds $qA$, where 
$$
A\defeq \sprod{p\le k\\p\,\nmid\,q}p.
$$
From now on, we suppose that
\be\label{eq:cong3}
e\equiv 0\bmod A
\ee
(since $\gcd(A,q)=1$, this is compatible with the
conditions \eqref{eq:cong1} and \eqref{eq:cong2}
already imposed on $e$). Let
\be\label{eq:bj-defn}
b_j\defeq b+Acdqj\qquad(j=1,\ldots,k),
\ee
and denote
$$
B\defeq b_1\cdots b_kcd.
$$
Our choices of $b,c,d,A$ ensure that the numbers
$b_1,\ldots,b_k,c,d$ are coprime in pairs, and 
$\gcd(B,q)=1$. Therefore, in addition to \eqref{eq:cong1}, \eqref{eq:cong2} and \eqref{eq:cong3},
using the CRT along with \eqref{eq:bj-defn}
we can further arrange for the integer $e$ to satisfy
the congruence conditions
\be\label{eq:ghost}
\begin{split}
cde+c+d&\equiv 0\bmod b_j,\\
b_jde+b_j+d&\equiv 0\bmod c,\\
b_jce+b_j+c&\equiv 0\bmod d,
\end{split}
\ee
for every $j$.

To finish the proof, let
$$
\cZ\defeq\{b_1,\ldots,b_k,c,d\},
$$
let $\cH(X)$ be the $(k+2)$-tuple
comprised of the linear forms
$$
F_z(X)\defeq z(e+ABqX)+1\qquad(z\in\cZ),
$$
and put
$$
f_\cH(X)\defeq \prod_{z\in\cZ} F_z(X).
$$
Under the Prime $k$-tuple Conjecture, the numbers
\dalign{
r_j&\defeq b_j(e+ABqm)+1\qquad(j=1,\ldots,k),\\
s&\defeq c(e+ABqm)+1,\\
t&\defeq d(e+ABqm)+1,
}
are simultaneously prime for infinitely many $m\in\N$ provided that
$f_\cH(X)$ is admissible and \eqref{eq:vandermonde}
holds (with $g_i,h_i$ suitably defined). Assuming this
for the moment, let $m$ be one such integer (fixed),
and let $n\defeq e+ABqm$. For each $j$, using \eqref{eq:ghost}
and the fact that $n\equiv e\bmod b_jcd$, 
Lemma~\ref{lem:Carm-criterion} shows that $r_jst$ is a
Carmichael number.
Since $r_1<\cdots<r_k$ is an arithmetic progression
(see \eqref{eq:bj-defn}), $r_1st<\cdots<r_kst$ is an
arithmetic progression of Carmichael numbers. Using
\eqref{eq:eureka} and \eqref{eq:bj-defn}, we also have
$$
r_jst\equiv a\bmod q\qquad(j=1,\ldots,k).
$$
Since $k$ is arbitrary, the theorem follows.

It remains to verify the conditions of the
Prime $k$-tuple Conjecture.

To see that $f_\cH(X)$
is admissible, observe that for any fixed $z\in\cZ$ the set
$$
\cS_z\defeq\{n\bmod p : z(e+ABqn)+1 \equiv 0 \bmod p\}
$$
has cardinality one if $p\nmid ABq$; for such primes we have
$p>k$ (since $p\nmid Aq$), and thus
$$
\big|\big\{n\bmod p : f_\cH(n)
\equiv 0 \bmod p\big\}\big|
=\big|{\textstyle\bigcup_z\cS_z}\big|=k<p,
$$
as required. On the other hand, for primes $p\mid ABq$ we
claim that
\be\label{eq:zedisdead}
ze+1\not\equiv 0\pmod p\qquad (z\in\cZ),
\ee
which implies that
$$
\big|\big\{n\bmod p : f_\cH(n)
\equiv 0 \bmod p\big\}\big|
=\big|{\textstyle\bigcup_z\cS_z}\big|=0<p.
$$
Indeed, if $p\mid q$, then \eqref{eq:zedisdead} is
a consequence of \eqref{eq:ball} and \eqref{eq:bj-defn}.
When $p\mid B$ (in other words, $p\in\cZ$), \eqref{eq:zedisdead}
follows from \eqref{eq:ghost}. If $p\mid A$, then
\eqref{eq:zedisdead} is implied by \eqref{eq:cong3}.

Finally, writing $F_z(X)=g_zX+h_z$
with $g_z\defeq zABq$ and $h_z\defeq ze+1$
for each $z\in\cZ$, we have
$$
g_{z_1}h_{z_2}-g_{z_2}h_{z_1}
=ABq(z_1-z_2)\ne 0\qquad(z_1,z_2\in\cZ,~z_1\ne z_2),
$$
and \eqref{eq:vandermonde} follows.
\end{proof}

\end{document}